\documentclass[12pt]{amsart}
\usepackage{url}
\usepackage{amsmath,amsthm,amssymb,amsfonts, stmaryrd}
\usepackage{fancybox}
\usepackage{color, graphicx}
\usepackage[T1]{fontenc} 
\usepackage[usenames,dvipsnames,svgnames,table]{xcolor}
\newtheorem{theorem}{Theorem}[section]

\newtheorem{conjecture}[theorem]{Conjecture}
\newtheorem{corollary}[theorem]{Corollary}

\newtheorem{definition}[theorem]{Definition}
\newtheorem{example}[theorem]{Example}

\newtheorem{lemma}[theorem]{Lemma}
\newtheorem{notation}[theorem]{Notation}

\newtheorem{proposition}[theorem]{Proposition}
\newtheorem{remark}[theorem]{Remark}

\DeclareMathOperator{\card}{card}

\newcommand{\Z}{\mathbb Z}

\newcommand{\N}{\mathbb N}

\newcommand{\fc}{\mathcal F}

\newcommand{\vs}[1]{\langle #1 \rangle}

\usepackage{tikz}

\begin{document}

\title{Gapsets of small multiplicity}

\author{S.~Eliahou and J.~Fromentin}

\address{Shalom Eliahou, Univ. Littoral C\^ote d'Opale, EA 2597 - LMPA - Laboratoire de Math\'ematiques Pures et Appliqu\'ees Joseph Liouville, F-62228 Calais, France and CNRS, FR 2956, France} \email{eliahou@univ-littoral.fr}

\address{Jean Fromentin, Univ. Littoral C\^ote d'Opale, EA 2597 - LMPA - Laboratoire de Math\'ematiques Pures et Appliqu\'ees Joseph Liouville, F-62228 Calais, France and CNRS, FR 2956, France} \email{fromentin@math.cnrs.fr}

\begin{abstract} A \emph{gapset} is the complement of a numerical semigroup in $\N$. In this paper, we characterize all gapsets of multiplicity $m \le 4$. As a corollary, we provide a new simpler proof that the number of gapsets of genus $g$ and fixed multiplicity $m \le 4$ is a nondecreasing function of $g$. 
\end{abstract}
\maketitle


\section{Introduction}

Denote $\N=\{0,1,2,3,\dots\}$ and $\N_+=\N\setminus \{0\}=\{1,2,3,\dots\}$. For $a,b \in \Z$, let $[a,b]=\{z \in \Z \mid a \le z \le b\}$ and $[a,\infty[=\{z \in \Z \mid a \le z\}$ denote the integer intervals they span. A \emph{numerical semigroup} is a subset $S \subseteq \N$ containing~$0$, stable under addition and with finite complement in $\N$. Equivalently, it is a subset $S \subseteq \N$ of the form $S = \vs{a_1,\dots,a_n}=\N a_1 + \dots + \N a_n$ for some globally coprime positive integers $a_1,\dots,a_n$.

\smallskip
For a numerical semigroup $S \subseteq \N$, its \emph{gaps} are the elements of $\N \setminus S$, its \emph{genus} is $g=|\N \setminus S|$, its \emph{multiplicity} is $m = \min S \setminus \{0\}$, its \emph{Frobenius number} is $f = \max \Z\setminus S$, its \emph{conductor} is $c=f+1$, and its \emph{embedding dimension}, usually denoted $e$, is the least number of generators of $S$, \emph{i.e.} the least $n$ such that $S = \vs{a_1,\dots,a_n}$. Note that the conductor $c$ of $S$ satisfies $c+\N \subseteq S$, and is minimal with respect to this property since $c-1=f \notin S$.

Given $g \ge 0$, the number $n_g$ of numerical semigroups of genus $g$ is finite, as easily seen. The values of $n_g$ for $g=0,\dots,15$ are as follows:
$$
1,1,2,4,7,12,23,39,67,118,204,343,592,1001,1693,2857.
$$ 
In 2006, Maria Bras-Amor\'os made some remarkable conjectures concerning the growth of $n_g$. In particular, she conjectured~\cite{Br08} that 
\begin{equation}\label{strong conjecture}
n_g  \,\ge\, n_{g-1}+n_{g-2}
\end{equation}
for all $g \ge 2$. This conjecture is widely open.
Indeed, even the weaker inequality
\begin{equation}\label{weak conjecture}
  n_g \,\ge\,  n_{g-1}
\end{equation}
whose validity has been settled by Alex Zhai \cite{Z} for all sufficiently large~$g$, remains to be proved for all $g \ge 1$.

In that same paper, Zhai showed that `most' numerical semigroups~$S$ satisfy $c \le 3m$, where $c$ and $m$ are the conductor and multiplicity of~$S$, respectively. For a more precise statement, let us denote by $n'_g$ the number of numerical semigroups of genus $g$ satisfying $c\le 3m$. The values of $n_g'$ for $g=0,\dots,15$ are as follows:
$$
1,1,2,4,6,11,20,33,57,99,168,287,487,824,1395,2351.
$$ 

Zhai showed then that $\lim_{g \to \infty} n'_g/n_g =1$, as had been earlier conjectured by Yufei Zhao \cite{Zhao}.
In that sense, numerical semigroups satisfying $c\leq 3m$ may be considered as \emph{generic}.
 
Recently, the strong conjecture~\eqref{strong conjecture} has been established for generic numerical semigroups. Here is the precise statement, first announced at the IMNS 2018 conference in Cortona~\cite{IMNS}.
\begin{theorem}[\cite{EF2}, Theorem 6.4]
The inequalities 
$$n'_{g-1}+n'_{g-2}+n'_{g-3}\geq n'_g\geq n'_{g-1}+n'_{g-2},$$
hold for all $g\geq 3$.
\end{theorem}

The proof of this result essentially rests on the notion of \emph{gapset filtrations}, a new flexible framework to investigate numerical semigroups introduced in~\cite{EF2}. More details are given in Section~\ref{section gapsets} since, here also, gapsets filtrations are at the core of the present results.
 
\begin{notation} Let $g\geq 0, m\geq 1$ be two integers.
We denote by~$\Gamma_{g,m}$ the finite set of all numerical semigroups of genus $g$ and multiplicity $m$, and by $n_ {g,m}=|\Gamma_{g,m}|$ its cardinality.
\end{notation}

Since, for a numerical semigroup $S$ of multiplicity $m$ and genus $g$, the integers $1,\dots,m-1$ belong to 
the complement $\N\setminus S$, the relation $g\geq m-1$ holds. Thus $n_{g,m}=0$ for $m\geq g+2$, and so we have
$$
n_g=\sum_{m=1}^{g+1} n_{g,m}.
$$
The first values of $n_{g,m}$ for $g\geq 0$ and small fixed $m$ are given below.
$$
\small
\begin{array}{r|rrrrrrrrrrrrrrrrrrr}
    g&0&1&2&3&4&5& 6& 7& 8& 9&10&11&12& 13& 14&\dots\\
  \hline
  m=1&1&0&0&0&0&0& 0& 0& 0& 0& 0& 0& 0&  0&  0&\dots\\
  m=2&0&1&1&1&1&1& 1& 1& 1& 1& 1& 1& 1&  1&  1&\dots\\
  m=3&0&0&1&2&2&2& 3& 3& 3& 4& 4& 4& 5&  5&  5&\dots\\
  m=4&0&0&0&1&3&4& 6& 7& 9&11&13&15&18& 20& 23&\dots\\
  m=5&0&0&0&0&1&4&7&10&13&16&22&24&32&35& 43&\dots\\
  m=6&0&0&0&0&0&1&5&11&17&27&37&49&66&85&106&\dots
\end{array}$$
For instance, the unique numerical semigroup of multiplicity $1$ is $\N$. Nathan Kaplan proposed the following conjecture in ~\cite{K1}, a refinement of the conjectured inequality \eqref{weak conjecture}. 
\begin{conjecture}\label{conjecture kaplan} Let $m\geq 2$. Then 
\begin{equation}\label{small conjecture}
n_{g,m}\geq n_{g-1,m}
\end{equation}
for all $g \ge 1$.
\end{conjecture}
On the other hand, still for $m \ge 2$ fixed, there is no hope a stronger inequality such as~\eqref{strong conjecture} may hold for the $n_{g,m}$, as the reader can check by looking at the rows of the above table.

Conjecture~\ref{conjecture kaplan} is trivial for $m=2$ since $n_{g,2}=1$ for all $g \ge 1$, and has been settled for $m=3,4,5$ in 2018 by Pedro A. Garc\'{\i}a-S\'anchez, Daniel Mar\'{i}n-Arag\'on and Aureliano M. Robles-P\'erez~\cite{GMR}. For that, they used a linear integer software to count the number of integral points of the associated Kunz polytope. With it, they first achieved formulas for $n_{g,m}$ for $m=3,4,5$, and then proved them to be increasing using a computer algebra system. The conjecture remains open for $m \ge 6$.
 
Our purpose in this paper is to give new proofs of Conjecture~\ref{conjecture kaplan} for $m=3$ and $m=4$ by constructing explicit injections 
$$\Gamma_{g,3} \to \Gamma_{g+1,3} \quad \textrm{and} \quad \ \Gamma_{g,4} \to \Gamma_{g+1,4}$$ for $g \ge 0$, thereby establishing the desired inequalities $n_{g+1,3}\geq n_{g,3}$ and $n_{g+1,4}\geq n_{g,4}$. Thus, our proofs are computer-free and do not rest on counting formulas for $n_{g,3}$ and $n_{g,4}$. These injections were first announced in~\cite{EF2}.

\section{Gapset filtrations}\label{section gapsets}
The content of this section is mostly taken from~\cite{EF2}.

\begin{definition} Let $n \in \N_+$. An \emph{additive decomposition} of $n$ is any expression of the form $n=a+b$ with $a,b \in \N_+$. We refer to the positive integers $a,b$ as the \emph{summands} of this decomposition.
\end{definition}

\begin{definition} A \emph{gapset} is a finite set $G \subset \N_+$ satisfying the following property: for all $z \in G$, if $z=x+y$ with $x,y \in \N_+$, then $x \in G$ or $y \in G$. That is, for any additive decomposition of $z \in G$, at least one of its summands belongs to $G$.  
\end{definition}
Notice the similarity of this definition with that of a prime ideal $P$ in a ring $R$, where for any $z\in P$, any decomposition $z=xy$ with $x,y \in R$ implies $x \in P$ or $y \in P$. 

\begin{remark} It follows from the definition that a gapset $G$ is nothing else than the set of gaps of a numerical semigroup $S$, where $S = \N \setminus G$.
\end{remark}

\begin{definition}
We naturally extend the definitions of \emph{multiplicity}, \emph{Frobenius number}, \emph{conductor} and \emph{genus} of a gapset $G$ as being those of the corresponding numerical semigroup $S = \N \setminus G$, respectively. 
\end{definition}
More directly, for a gapset $G$, these notions may be described as follows:
\begin{enumerate}
\item[---]  the multiplicity of $G$ is the smallest integer $m\geq 1$ such that  $m\not\in G$;

\item[---]  the Frobenius number of $G$ is $\max(G)$ if $G\not=\emptyset$, and $-1$ otherwise;

\item[---]  the conductor of $G$ is $1+\max(G)$ if $G\not=\emptyset$, and $0$ otherwise;

\item[---]  the genus of $G$ is $g(G)=\card(G)$.
\end{enumerate}

\begin{example}
\label{Ex2}
The set $G=\{1,2,3,4,6,7,11\}$ is a gapset. For instance, for each additive decomposition of $11$, namely 
$$1+10, \quad 2+9,\quad 3+8, \quad 4+7, \quad 5+6,$$ 
at least one of the two summands belongs to $G$. Let $S=\N\setminus G=\{0,5,8,9,10\}\cup[12,+\infty[.$ Then $S=\left<5,8,9\right>$ as easily seen, whence $S$ is indeed a numerical semigroup. The multiplicity, conductor, Frobenius number, genus and embedding dimension of $G$ and $S$ are $m=5$, $c=11$, $f=12$, $g=7$ and $e=3$, respectively.
\end{example}

\subsection{The canonical partition}

\begin{lemma} Let $G$ be a gapset of multiplicity $m$. Then 
\begin{eqnarray*}
[1,m-1] & \subseteq & G, \\
G \cap m\N & = & \emptyset.
\end{eqnarray*}
\end{lemma}
\begin{proof}  By definition of the multiplicity, $G$ contains $[1,m-1]$ but not~$m$. Let $a \ge 2$ be an integer. The formula $am=m+(a-1)m$ and induction on $a$ imply that $am \notin G$.
\end{proof}
 This motivates the following notation and definition.

\begin{notation} Let $G$ be a gapset of multiplicity $m$. We denote $G_0=[1,m-1]$ and, more generally,
\begin{equation}\label{G_i}
G_i = G \cap [im+1,(i+1)m-1] \quad \mbox{ for all } i \ge 0.
\end{equation}
\end{notation}

\begin{definition}
Let $G$ be a gapset of multiplicity $m$ and conductor $c$. 
The \emph{depth} of $G$ is the integer~$q=\lceil c/m\rceil$.
\end{definition}

\begin{proposition}\label{subset} Let $G$ be a gapset of multiplicity $m$ and depth $q$. Let $G_i$ be defined as in \eqref{G_i}. Then
\begin{equation}\label{partition}
G = G_0 \sqcup G_1 \sqcup \dots \sqcup G_{q-1}
\end{equation}
and $G_{q-1} \not= \emptyset$. Moreover $G_{i+1} \subseteq m+G_{i}$ for all $i \ge 0$. 
\end{proposition}
\begin{proof} As $G \cap m\N = \emptyset$, it follows that $G$ is the disjoint union of the $G_i$ for $i \ge 0$. Let $c$ be the conductor of $G$. Then $G \subseteq [1,c-1]$. Since $(q-1)m < c \le qm$ by definition of $q$, it follows that $G_i = \emptyset$ for $i \ge q$, whence \eqref{partition}. Let $f=c-1$. Since $f \in G$, $(q-1)m \le f < qm$ and $f \not\equiv 0 \bmod m$, it follows that $f \in G_{q-1}$. 

It remains to show that $G_{i+1} \subseteq m+G_{i}$ for all $i \ge 0$.
Let $x \in G_{i+1}$. Since $G_{i+1} \subseteq [(i+1)m+1, (i+2)m-1]$, we have 
$$
x-m \in [im+1, (i+1)m-1].  
$$
Now $x-m \in G$ since $x = m + (x-m)$ and $m \notin G$. So $x-m \in G_i$.
\end{proof}
\begin{definition} Let $G$ be a gapset. The \emph{canonical partition} of $G$ is the partition $G = G_0 \sqcup G_1 \sqcup \dots \sqcup G_{q-1}$ given by Proposition~\ref{subset}.
\end{definition}

\begin{remark}\label{data gapsets}
The multiplicity $m$, genus $g$ and depth $q$ of a gapset $G$ may be read off from its canonical partition $G = \sqcup_i G_i$ as follows : 
\begin{eqnarray*}
m & = & \max(G_0) +1, \\
g & = & \sum_i |G_i|, \\
q & = & \textrm{the number of parts of the partition}.
\end{eqnarray*}
\end{remark}

\subsection{Gapset filtrations} 

Let $G \subset \N_+$ be a gapset. Let $G = G_0 \sqcup G_1 \sqcup \dots \sqcup G_{q-1}$ be its canonical partition. For all $0 \le i \le q-1$, denote
\begin{equation}\label{Fi Gi}
F_i = -im +G_i.
\end{equation}
Then $F_{i+1} \subseteq F_i$ for all $i$, as follows from the inclusion $G_{i+1} \subseteq m+G_{i}$ stated in Proposition~\ref{subset}. This gives rise to the following definition.

\begin{definition}
Let $G \subset \N_+$ be a gapset of multiplicity $m$ and depth $q$. The \emph{gapset filtration} associated to $G$ is the finite sequence 
$$(F_0,F_1,\dots,F_{q-1})=(G_0,-m+G_1,\dots,-(q-1) m +G_{q-1}),$$
\end{definition}
\noindent
i.e. with $F_i$ defined as in \eqref{Fi Gi} for all $i$.Thus, as seen above, we have  
\begin{equation}
\label{decr_filtration}
F_0=[1,m-1]\supseteq F_1 \supseteq \dots \supseteq F_{q-1}.
\end{equation}

\smallskip
We define the \emph{multiplicity}, \emph{Frobenius number}, \emph{conductor} and \emph{genus} of a gapset filtration $F=(F_0,\dots,F_{q-1})$ from those of the corresponding gapset G, namely:
\begin{itemize}
\item[---] the multiplicity of $F$ is $1+\max(F_0)$ if $F_0\not=\emptyset$ and $0$ otherwise;

\item[---] the Frobenius number of $F$ is $qm+\max(F_{q-1})$ if $F_0\not=\emptyset$ and $-1$ otherwise;

\item[---] the conductor of $F$ is $1+qm+\max(F_{q-1})$ if $F_0\not=\emptyset$ and $0$ otherwise;

\item[---] the genus of $F$ is $\card(F_0)+\dots+\card(F_{q-1})$.
\end{itemize}

\begin{example}
\label{Ex3}
Consider the gapset $G=\{1,2,3,4,6,7,11\}$ of Example~\ref{Ex2}. Its multiplicity is $m=5$, and its canonical partition is given by $G_0=\{1,2,3,4\}$, $G_1=\{6,7\}$ and $G_2=\{11\}$. Thus, its associated filtration is $F=(\{1,2,3,4\},\{1,2\},\{1\})$.
\end{example}

\begin{definition}
For integers $g\geq 1, m\geq 1$, we denote by $\fc(g, m)$ the set of all gapset filtrations of genus $g$ and multiplicity $m$.
\end{definition} 

Note that any given gapset filtration $F=(F_0, \dots, F_{q-1})$ corresponds to a \emph{unique} gapset $G$, since \eqref{Fi Gi} is equivalent to 
\begin{equation}\label{Gi Fi}
G_i = im +F_i.
\end{equation}
In particular, \emph{there is a straigthforward bijection between gapsets $G$ and gapset filtrations $F$}, which naturally preserves the multiplicity, Frobenius number, conductor and genus. Here is a direct consequence.

\begin{proposition}
\label{card_fc}
For any integers $g\geq 1, m\geq 1$, we have $$n_{g,m} = |\fc(g,m)|.$$
\end{proposition}
\begin{proof} Straightforward from the above discussion.
\end{proof}

This result allows us to study properties of the sequence $g\mapsto n_{g,m}$ in the setting of gapset filtrations of multiplicity $m$. In particular, in order to establish its growth, it suffices to exhibit injections from $\fc(g,m)$ to $\fc(g+1,m)$. This is what we achieve  in subsequent sections for $m=3$ and $m=4$. 

We start with the separate case $m=3$, which can be treated in a straightforward way and which points to a general strategy for larger values of $m$. Then, following those clues, we introduce some general tools, and we end up applying them to the case $m=4$.

\section{The case $m=3$}\label{section m=3}

Any filtration $(F_0,\dots,F_t)$ such that 
$$\{1,2\}=F_0 \supseteq F_1 \supseteq \dots \supseteq F_{t}\not= \emptyset$$ is of one of the two possible forms below, with the terms on the left standing as a compact notation:

\begin{eqnarray*}
(12)^r(1)^s & = & (\underbrace{\{1,2\},\dots,\{1,2\}}_{r},\underbrace{\{1\},\dots,\{1\}}_s), \\
(12)^r(2)^s & = & (\underbrace{\{1,2\},\dots,\{1,2\}}_{r},\underbrace{\{2\},\dots,\{2\}}_s),
\end{eqnarray*} 
both with $r \ge 1$ since $F_0=\{1,2\}$, and $s \ge 0$. We now characterize those filtrations which are gapset filtrations of multiplicity $3$.

\begin{theorem}
\label{thm m=3}
The gapset filtrations of multiplicity $m=3$ are exactly the following ones:
\begin{align*}
(12)^r(2)^s&\quad\text{with $0 \le s\leq r$},\\
(12)^r(1)^s&\quad\text{with $0 \le s\leq r+1$},
\end{align*}
both with $r\geq 1$.
\end{theorem}
Note that $g=2r+s$ in both cases, since the genus of a gapset filtration $F=(F_0,\dots,F_{q-1})$ is given by the sum of the $|F_i|$.

\begin{proof} We start with the second case.

\smallskip
\noindent
\underline{Case $F=(12)^r(2)^s$.} Then 
\begin{eqnarray*}
& F_0 = \dots=F_{r-1}=\{1,2\}, \\
& F_r = \dots=F_{r+s-1}=\{2\}.
\end{eqnarray*}
Using \eqref{Gi Fi} with $m=3$, namely $G_i =3i+F_i$ for all $i$, let 
\begin{equation}\label{set G}
G = G_0 \cup \dots \cup G_{r+s-1}
\end{equation} 
be the corresponding finite set. By construction, $F$ is a gapset filtration if and only if $G$ a gapset. So, when is it the case that $G$ is a gapset? We now proceed to answer this question.

\smallskip
\noindent
\textbf{Step 1.} The set $G$ given by \eqref{set G} has the following properties:
\begin{eqnarray*}
3\N \cap G & = & \emptyset \\
3i+1 \in G & \iff & i \le r-1\\
3i+2 \in G & \iff & i \le r+s-1.
\end{eqnarray*}
Indeed, this directly follows from the definition $G_i=3i+F_i$ and \eqref{set G}.

\smallskip
\noindent
\textbf{Step 2.} For $i \in \N$, any additive decomposition $3i+1=a+b$ is of the form 
$$
(a,b) = (3x+1, 3(i-x)) \mbox{ or}\,\, (3y+2, 3(i-1-y)+2)
$$
for some integers $0 \le x \le i-1$ or $0 \le y \le i-1$. Similarly, any additive decomposition $3i+2=a+b$
is of the form
$$
(a,b) = (3x+2, 3(i-x)) \mbox{ or}\,\,  (3y+1, 3(i-y)+1)
$$
for some integers $0 \le x \le i-1$ or $0 \le y \le i$. 

\smallskip
\noindent
\textbf{Step 3.} 
Let $3i+1 \in G$, i.e. with $i \le r-1$ according to Step 1. We now show that for any additive decomposition $3i+1=a+b$, either $a$ or $b$ belongs to $G$. Using Step 1, if $(a,b) = (3x+1, 3(i-x))$, then $a \in G$ since $x \le i$ and we are done. Similarly, if $(a,b)=(3y+2, 3(i-1-y)+2)$, then $a \in G$ since $y \le i \le r-1 \le r+s-1$ and we are done again.

\smallskip
\noindent
\textbf{Step 4.} Let $3i+2 \in G$, i.e. with $i \le r+s-1$. Let $3i+2=a+b$ be any additive decomposition. If $(a,b) = (3x+2, 3(i-x))$, then $a \in G$ since $x \le i$ and we are done.   Assume now $(a,b) = (3y+1, 3(i-y)+1)$ with $0 \le y \le i$. Then $a,b \notin G$ if and only if $y,i-y \ge r$. This is only possible if $i \ge 2r$ and, since $i \le r+s-1$ by hypothesis, the latter is equivalent to $s-1 \ge r$. In particular, if $s \le r$, then either $a$ or $b$ belongs to $G$. In summary, we have
$$
(12)^r(2)^s \mbox{ is a gapset filtration } \iff G \mbox{ is a gapset } \iff s \le r,  \vspace{-0.07cm}
$$
as desired.

\medskip
\noindent
\underline{Case $F=(12)^r(1)^s$.} The arguments are similar to those of the previous case. Here, to start with, we have
\begin{eqnarray*}
& F_0 = \dots=F_{r-1}=\{1,2\}, \\
& F_r = \dots=F_{r+s-1}=\{1\}.
\end{eqnarray*}
The corresponding set $G$ defined by $G_i =3i+F_i$ for all $i$ and \eqref{Gi Fi} has the following properties:
\begin{eqnarray*}
3\N \cap G & = & \emptyset \\
3i+1 \in G & \iff & i \le r+s-1\\
3i+2 \in G & \iff & i \le r-1.
\end{eqnarray*}
Analogously to Step 3 above, it is easy to see that for any additive decomposition $a+b=3i+2$ where $3i+2 \in G$, then either $a$ or $b$ belongs to $G$.

On the other hand, let $3i+1 \in G$. Then, analogously to Step 4 above, we find that there exists an additive decomposition $3i+1=a+b$ with $a,b \notin G$ if and only if $s \ge r+2$. The details, using Step 2 and the above properties of $G$, are straightforward and left to the reader. Therefore, $G$ is a gapset if and only if $s \le r+1$, as claimed. This concludes the proof of the proposition.  
\end{proof}

Here is a straightforward consequence of the above characterization and the main result of this section.

\begin{corollary}
\label{T3} For all $g \ge 0$, there is a natural injection
$$\fc(g, 3) \longrightarrow \fc(g+1, 3).$$
In particular, we have $n_{g+1,3} \ge n_{g,3}$ for all $g \ge 0$.
\end{corollary}

\begin{proof}
Since $\fc(g, 3) = \emptyset$ for $g \le 1$, the statement holds in this case. Assume now $g \ge 2$. For $F=(F_0,\dots,F_{q-1}) \in \fc(g, 3)$, let us denote by $f_1(F)$ the insertion of a $1$ in $F$ at the unique possible position to get a new nonincreasing sequence of subsets of $[1,2]$. That is, for $r,s \ge 1$, we define
\begin{align*} 
(12)^r & \stackrel{f_1}{\longmapsto} (12)^r(1) \\
(12)^r(1)^s & \stackrel{f_1}{\longmapsto} (12)^r(1)^{s+1} \\
(12)^r(2)^s & \stackrel{f_1}{\longmapsto} (12)^{r+1}(2)^{s-1}.
\end{align*}
When is it the case that $f_1(F)$ is still a \emph{gapset} filtration, of course automatically of genus $g+1$? In other words, when do we have that $f_1(F)$ belongs $\fc(g+1, 3)$? Theorem~\ref{thm m=3} easily provides the following answer.
\begin{itemize}
\item If $F=(12)^r(2)^s \in \fc(g,3)$, then $f_1(F) \in \fc(g+1,3)$ for all $r,s$.
\item If $F=(12)^r(1)^s \in \fc(g,3)$, then $f_1(F) \in \fc(g+1,3)$ if and only if $s \le r$.
\end{itemize}
Recall that $g=2r+s$ in both cases. In particular, the only case where $F \in \fc(g,3)$ but $f_1(F) \notin \fc(g+1,3)$ is for $F=(12)^r(1)^s$ with $s=r+1$, i.e. for $F=(12)^r(1)^{r+1} \in \fc(g,3)$ where $g=3r+1$. 

Consequently, $f_1$ provides a well-defined map
$$f_1 \colon \fc(g, 3) \longrightarrow \fc(g+1, 3),$$
obviously injective by construction, whenever $g \not\equiv 1 \bmod 3$.

\smallskip
Similarly, for $F \in \fc(g, 3)$, denote by $f_2(F)$ the insertion of a $2$ in $F$ where it makes sense. That is, for $r,s \ge 1$, define
\begin{align*}
(12)^r & \stackrel{f_2}{\longmapsto} (12)^r(2) \\
(12)^r(1)^s & \stackrel{f_2}{\longmapsto} (12)^{r+1}(1)^{s-1} \\
(12)^r(2)^s & \stackrel{f_2}{\longmapsto} (12)^r(2)^{s+1}.
\end{align*}
By Theorem~\ref{thm m=3} again, we have
\begin{itemize}
\item If $F=(12)^r(2)^s \in \fc(g,3)$, then $f_2(F) \in \fc(g+1,3)$ if and only if $s \le r-1$.
\item If $F=(12)^r(1)^s \in \fc(g,3)$, then $f_2(F) \in \fc(g+1,3)$ for all $r,s \ge 1$.
\end{itemize}
In particular, the only case where $F \in \fc(g,3)$ but $f_2(F) \notin \fc(g+1,3)$ is for $F=(12)^r(2)^r \in \fc(g,3)$ with $g=3r$. Therefore, $f_2$ provides a well-defined injective map 
$$f_2 \colon \fc(g, 3) \longrightarrow \fc(g+1, 3)$$
whenever $g \not \equiv 0 \bmod 3$.

\smallskip
Summarizing, we end up with a well-defined injective map
$$f \colon \fc(g, 3) \longrightarrow \fc(g+1, 3)$$
defined by $f=f_1$ if $g \equiv 0,2 \bmod 3$, and $f=f_2$ otherwise.
\end{proof}

\section{Some more general tools}

In order to facilitate discussing gapsets and gapset filtrations, and gather more tools to treat more cases, it is useful to consider somewhat more general subsets of $\N_+$. 

\subsection{On $m$-extensions and $m$-filtrations}\label{m-extension}

\begin{definition} Let $m \in \N_+$. An \emph{$m$-extension} is a finite set $A \subset \N_+$ containing $[1,m-1]$ and admitting a partition
\begin{equation}\label{m-ext}
A = A_0 \sqcup A_1 \sqcup \dots \sqcup A_{t}
\end{equation}
for some $t \ge 0$, where $A_0=[1,m-1]$ and $A_{i+1} \subseteq m+A_i$ for all $i \ge 0$. 
\end{definition}
In particular, an $m$-extension $A$ satisfies $A \cap m\N = \emptyset$. Moreover, the  above conditions on the $A_i$ imply 
\begin{equation}\label{A_i}
A_i=A \cap [im+1, (i+1)m-1]
\end{equation}
for all $i \ge 0$, whence the $A_i$ are \emph{uniquely determined} by $A$.

\begin{remark}\label{gapset is extension}
Every gapset of multiplicity $m$ is an $m$-extension. This follows from Proposition~\ref{subset}.
\end{remark}

\smallskip
Closely linked is the notion of \emph{$m$-filtration}.
\begin{definition}\label{filtration} Let $m \in \N_+$. An \emph{$m$-filtration} is a finite sequence $F=(F_0,F_1,\dots,F_t)$ of nonincreasing subsets of $\N_+$ such that
$$
F_0=[1,m-1] \supseteq F_1 \supseteq \dots \supseteq F_t.
$$
The \emph{genus} $g$ of $F$ is defined as $g = \sum_{i=0}^t |F_i|$.
 \end{definition}
For $m \in \N_+$, there is a straightforward bijection between $m$-extensions and $m$-partitions.

\begin{proposition}\label{fil-ext} Let $A=A_0 \sqcup A_1 \sqcup \dots \sqcup A_{t}$ be an $m$-extension. Set $F_i=-im+A_i$ for all $i$. Then
$(F_0,F_1,\dots,F_t)$ is an $m$-filtration. Conversely, let $(F_0,F_1,\dots,F_t)$ be an $m$-filtration. Set $A_i = im+F_i$ for all $i$, and let 
$$A = \bigsqcup_{i=0}^{t} A_i = \bigsqcup_{i=0}^{t} (im+F_i).$$
Then $A$ is an $m$-extension.
\end{proposition}
\begin{proof} We have $F_i=-im+A_i$ if and only if $A_i=im+F_i$.
\end{proof}

\begin{notation}\label{phi and tau} If $A$ is an $m$-extension, we denote by $F=\varphi(A)$ the $m$-filtration associated to it by Proposition~\ref{fil-ext}. Conversely, if $F$ is an $m$-filtration, we denote by $A=\tau(F)$ its associated $m$-extension. 
\end{notation}

By Proposition~\ref{fil-ext}, \emph{the maps $\varphi$ and $\tau$ are inverse to each other}.

\subsection{Gapset filtrations revisited} 

\begin{definition}
Let $G \subset \N_+$ be a gapset of multiplicity $m$. The \emph{gapset filtration} associated to $G$ is the $m$-filtration $F=\varphi(G)$.
\end{definition}
By Remark~\ref{gapset is extension}, every gapset $G$ of multiplicity $m$ is an $m$-extension, whence $\varphi(G)$ is well-defined. 

\smallskip
Concretely, let $G$ be a gapset of multiplicity $m$ and depth $q$. As in \eqref{G_i}, let $G_i = G \cap [im+1,(i+1)m-1]$ for all $i \ge 0$, so that $G_0=[1,m-1]$ and
$$
G = G_0 \sqcup \dots \sqcup G_{q-1}.
$$
The associated $m$-filtration $F=\varphi(G)$ is then given by $F=(F_0,\dots,F_{q-1})$ where $F_i = -im+G_i$ for all $i \ge 0$.

It follows from Remark~\ref{data gapsets} and the equality $|F_i|=|G_i|$ for all $i$, that the genus of $F$ is equal to $|F_0|+\dots+|F_{q-1}|$ and that its depth is equal to the number of nonzero $F_i$.

\subsection{A compact representation}

In this section, we use permutations of $[1,m-1]$ and exponent vectors to represent $m$-filtrations in a compact way. We denote by $\mathfrak{S}_{m-1}$ the symmetric group on $[1,m-1]$.

\begin{proposition}\label{compact} Let $F=(F_0,\dots,F_t)$ be an $m$-filtration. Then there exists a permutation $\sigma \in \mathfrak{S}_{m-1}$ and exponents $e_0,\dots, e_{m-2} \in \N$ such that
$$
F= (\underbrace{F'_0,\dots,F'_0}_{e_0}, \underbrace{F'_1,\dots,F'_1}_{e_1}, \dots , \underbrace{F'_{m-2},\dots,F'_{m-2}}_{e_{m-2}}),
$$
where $F'_0=[1,m-1]$ and $F'_i=F'_{i-1}\setminus\{\sigma(i)\}$ for $1 \le i \le m-2$. In particular, we have $|F'_i|=m-1-i$ for all $0 \le i \le m-2$.
\end{proposition}
\begin{proof} By hypothesis, we have
$$
[1,m-1] = F_0 \supseteq F_1 \supseteq \dots \supseteq F_t.
$$
Equalities may occur in this chain. Removing repetitions, let
$$
[1,m-1]=H_0 \supsetneq H_1 \supsetneq \dots \supsetneq H_s 
$$
denote the underlying descending chain, i.e. with 
$$\{F_0,F_1,\dots,F_t\} = \{H_0,H_1,\dots,H_s\}$$
and $H_i \not= H_j$ for all $i \not= j$. Each $H_i$ comes with some repetition frequency $\mu_i \ge 1$ in $\{F_0,F_1,\dots,F_t\} $. Thus, we have
$$
F = (\underbrace{H_0,\dots,H_0}_{\mu_0}, \underbrace{H_1,\dots,H_1}_{\mu_1}, \dots , \underbrace{H_{s},\dots,H_{s}}_{\mu_{s}}).
$$ 
Now, between each consecutive pair $H_{i-1} \supsetneq H_i$, we insert some maximal descending chain of subsets $H'_{i,j}$, i.e.
$$
H_{i-1} = H'_{i,0} \supsetneq H'_{i,1} \supsetneq \dots \supsetneq H'_{i,k_i} = H_i,
$$
where $k_i=|H_{i-1}| - |H_i|$. Thus $|H'_{i,j}|=|H'_{i-1}|-j$ for all $0 \le j \le k_i$.

We end up with a maximal descending chain of subsets
$$
F'=[1,m-1]=F'_0 \supsetneq F'_1 \supsetneq \dots \supsetneq F'_{m-2},
$$
where each term has one less element than the preceding one, i.e. where $|F'_j| = |F'_{j-1}|-1$ for all $1 \le j \le m-2$. By construction, we have
$$
\{F_0,F_1,\dots,F_t\} = \{H_0,H_1,\dots,H_s\} \subseteq \{F'_0, F'_1, \dots, F'_{m-2}\},
$$
and each $F'_i$ arises with some frequency $e_i \ge 0$ in $\{F_0,F_1,\dots,F_t\}$. Thus
$$
F = (\underbrace{F'_0,\dots,F'_0}_{e_0}, \underbrace{F'_1,\dots,F'_1}_{e_1}, \dots , \underbrace{F'_{m-2},\dots,F'_{m-2}}_{e_{m-2}}).
$$ 
Finally, since each $F'_i$ is obtained by removing one distinct element from $F'_{i-1}$ for $1 \le i \le m-2$, there is a permutation $\sigma$ of $[1,m-1]$ such that 
$$F'_i=F'_{i-1}\setminus\{\sigma(i)\}$$ 
for $1 \le i \le m-2$.
\end{proof}

\begin{notation}\label{F'} Given $\sigma \in \mathfrak{S}_{m-1}$ and $e=(e_0,\dots,e_{m-2}) \in \N^{m-1}$ such that $e_0 \ge 1$, we denote by $F(\sigma,e)$ the $m$-filtration
$$
F = (\underbrace{F'_0,\dots,F'_0}_{e_0}, \underbrace{F'_1,\dots,F'_1}_{e_1}, \dots , \underbrace{F'_{m-2},\dots,F'_{m-2}}_{e_{m-2}})
$$
where $F'_i=F'_{i-1}\setminus\{\sigma(i)\}$ for $1 \le i \le m-2$.
\end{notation}

\begin{example}
\label{Ex4}
Consider the $5$-filtration $F=(\{1,2,3,4\},\{1,2\},\{1\})$ of Example~\ref{Ex3}. Let $\sigma=(3,4,2,1) \in \mathfrak{S}_4$ and $e=(1,0,1,1)$. Then 
$F=F(\sigma,e)$ as readily checked. Note that we also have $F=F(\sigma',e)$ where $\sigma'=(4,3,2,1)$.
\end{example}

One important question is: when is the $m$-filtration $F=F(\sigma,e)$ a \emph{gapset} filtration? The next section provides an answer.

\subsection{Complementing an $m$-extension}

\begin{notation}
Let $F=F(\sigma,e)$ be an $m$-filtration, where $\sigma \in \mathfrak{S}_{m-1}$ and $e=(e_0,\dots,e_{m-2}) \in \N^{m-1}$ with $e_0 \ge 1$. We denote by $G=G(\sigma,e)$ the corresponding $m$-extension, i.e. $G = \tau(F)$ using Notation~\ref{phi and tau}.
\end{notation}

Here is how to determine the set complement in $\N$ of the $m$-extension $G=G(\sigma,e)$.

\begin{proposition}
\label{complement of extension}
Let $F=F(\sigma,e)$ be an $m$-filtration, where $\sigma \in \mathfrak{S}_{m-1}$ and $e=(e_0,\dots,e_{m-2}) \in \N^{m-1}$ with $e_0 \ge 1$. Let $G=G(\sigma,e)$ be the corresponding $m$-extension, i.e. $G = \tau(F)$. Then
\begin{equation}
\label{SF}
\N \setminus G=\bigsqcup_{i=0}^{m-1} \sigma(i)+m(e_0+\dots+e_{i-1}+\N),
\end{equation}
with the conventions $\sigma(0)=0$ and $e_0+\dots+e_{i-1}=0$ for $i=0$.
\end{proposition}
\begin{proof} For $0 \le i \le m-1$, denote $F_i=[1,m-1] \setminus \{\sigma(0), \dots, \sigma(i)\}$. Thus $F_0=[1,m-1]$, $F_1=[1,m-1] \setminus \{\sigma(1)\}$, $F_2=[1,m-1] \setminus \{\sigma(1), \sigma(2)\}$ and so on. By definition of $F=F(\sigma, e)$, we have
$$
F = (\underbrace{F_0,\dots,F_0}_{e_0}, \underbrace{F_1,\dots,F_1}_{e_1}, \dots , \underbrace{F_{m-2},\dots,F_{m-2}}_{e_{m-2}}).
$$
Let $G=\tau(F)$. For $k \in [0,m-1]$, set
$\displaystyle
G^{(k)} = \{x \in G \mid x \equiv k \bmod m\}.
$
Then $\displaystyle G = \bigsqcup_{k=0}^{m-1} G^{(k)}$. Since $G$ is an $m$-extension, we have $G \cap m\N = \emptyset$, i.e. $G^{(0)}=\emptyset$. We now proceed to determine $G^{(k)}$ for $k \ge 1$.  
Since $\sigma$ is a permutation of $[1,m-1]$, there exists $i \in [1,m-1]$ such that $k=\sigma(i)$. We claim that
\begin{equation}\label{G^k}
G^{(k)} = G^{(\sigma(i))} = \sigma(i)+m[0, e_0+\dots + e_{i-1}-1].
\end{equation}
Indeed by construction, for all $r \ge 0$ we have 
\begin{equation}\label{F_r}
\sigma(i) \in F_r \Leftrightarrow r \le i-1.
\end{equation}
Now, by definition of the map $\tau$, we have
\begin{eqnarray}\label{G}
G & = & \bigsqcup_{l=0}^{m-2} \big(   \bigsqcup_{j=e_0+\dots+e_{l-1}}^{e_0+\dots+e_{l}-1} (jm+F_l)  \big).
\end{eqnarray}
It follows from \eqref{F_r} and \eqref{G} that 
$$
\sigma(i)+jm \in G \Leftrightarrow j < e_0+\dots + e_{i-1}
$$
for all $j \ge 0$. This proves \eqref{G^k}. Taking the complement in $\N$, it follows that 
$$
\sigma(i)+jm \in \N \setminus G \Leftrightarrow j \ge e_0+\dots + e_{i-1}.
$$
This proves \eqref{SF}.
\end{proof}

\begin{notation}\label{S(sigma,e)} Given $\sigma \in \mathfrak{S}_{m-1}$ and $e=(e_0,\dots,e_{m-2}) \in \N^{m-1}$ with $e_0 \ge 1$, we denote
$$
S(\sigma, e) = \bigsqcup_{i=0}^{m-1} \sigma(i)+m(e_0+\dots+e_{i-1}+\N).
$$
\end{notation}
Thus, the above proposition amounts to the statement
$$
\N = G(\sigma, e) \sqcup S(\sigma, e)
$$
for all $\sigma \in \mathfrak{S}_{m-1}$ and $e=(e_0,\dots,e_{m-2}) \in \N^{m-1}$ with $e_0 \ge 1$.

This yields the following way to construct all gapsets of given multiplicity $m \ge 3$.

\begin{proposition}\label{S(sigma,e)} Let $m \ge 2$. Every numerical semigroup $S$ of multiplicity $m$ is of the form $S=S(\sigma,e)$ for some $\sigma\in\mathfrak{S}_{m-1}$ and $e=(e_0,\dots,e_{m-2})\in \N^{m}$ with $e_0 \ge 1$.
\end{proposition}
\begin{proof} Let $S$ be a numerical semigroup of multiplicity $m$. Let $G = \N \setminus S$ and $F=\varphi(G)$ be the associated gapset and gapset filtration, respectively. Then $F$ is an $m$-filtration, whence by Proposition~\ref{compact}, it is of the form $F=F(\sigma,e)$ for some $\sigma$ and $e$ of the desired type. Then $G=\tau(F)=G(\sigma,e)$, whence $S=\N \setminus G=S(\sigma,e)$. 
\end{proof}

\smallskip
We now determine the conditions under which a set of the form $S(\sigma,e)$ is a numerical semigroup.

\begin{theorem}
\label{thrm}
Let $m \ge 3$. Let $\sigma\in\mathfrak{S}_{m-1}$ and $e=(e_0,\dots,e_{m-2})\in \N^{m}$ with $e_0 \ge 1$. Then $S(\sigma,e)$ is a numerical semigroup if and only if for all $1 \le i,j,k \le m-1$ with $i \le j < k$, we have
\begin{equation*}
e_{j}+\cdots+e_{k-1} \le
\left\{
\begin{array}{ll}
e_0+\cdots+e_{i-1}  & \mbox{if }\, \sigma(i)+\sigma(j) = \sigma(k),\\
e_0+\cdots+e_{i-1}+1 & \mbox{if }\, \sigma(i)+\sigma(j) = \sigma(k)+m.
\end{array}
\right.
\end{equation*}
\end{theorem}

\begin{proof} Denote $S_0=\N$ and $S_i=\sigma(i)+m(e_0+\dots+e_{i-1}+\N)$ for $1 \le i \le m-1$. Let $S'=S(\sigma,e)$. Then
$$
S' = \bigsqcup_{i=0}^{m-1} S_i
$$
by definition. We have $0 \in S_0 \subset S'$. The complement of $S'$ in $\N$ is finite, since $N \setminus S(\sigma,e)=G(\sigma,e)$. It remains to prove that $S'$ is stable under addition if and only if the stated inequalities are satisfied.

Let $i,j$ be integers such that $0 \le i \le j \le m-1$. If $i=0$ then $S_i+S_j=S_j+m\N=S_j$.
We now assume $i\not=0$. There are three cases.

\smallskip
\textbf{$\bullet$ Case} $\sigma(i)+\sigma(j)\leq m-1$. There exists $k\in [1,m-1]$ satisfying  $\sigma(k)=\sigma(i)+\sigma(j)$. Then
\begin{align*}
S_i+S_j&=\sigma(i)+m(e_0+\dots+e_{i-1}+\N)+\sigma(j)+m(e_0+\dots+e_{j-1}+\N)\\
&=\sigma(k)+m(e_0+\dots+e_{i-1}+e_0+\dots+e_{j-1}+\N).
\end{align*}
Therefore $S_i+S_j$ is contained in $S'$ if and only if it is contained in $S_k$, and this occurs if and only if 
$$
e_0+\dots+e_{k-1}\leq e_0+\dots+e_{i-1}+e_0+\dots+e_{j-1}.
$$
This condition is plainly satisfied if $k < j$, and is equivalent to
$$
e_j+\dots+e_{k-1}\leq e_0+\dots+e_{i-1}
$$
if $k > j$.

\smallskip
\textbf{$\bullet$ Case} $\sigma(i)+\sigma(j)\geq m+1$. There exists $k\in [1,m-1]$ satisfying  $\sigma(k)+m=\sigma(i)+\sigma(j)$. Then
\begin{align*}
S_i+S_j&=\sigma(i)+m(e_0+\dots+e_{i-1}+\N)+\sigma(j)+m(e_0+\dots+e_{j-1}+\N)\\
&=\sigma(k)+m(e_0+\dots+e_{i-1}+e_0+\dots+e_{j-1}+1+\N).
\end{align*}
Again, $S_i+S_j$ is contained in $S'$ if and only if it is contained in $S_k$, and this occurs if and only if 
$$
e_0+\dots+e_{k-1}\leq e_0+\dots+e_{i-1}+e_0+\dots+e_{j-1}+1.
$$
This is plainly satisfied if $j < k$, and is equivalent to 
$$e_j+\dots+e_{k-1}\leq e_0+\dots+e_{i-1}+1$$ otherwise.

\smallskip
\textbf{$\bullet$ Case} $\sigma(i)+\sigma(j)=m$. Then $S_i+S_j\subseteq m\N=S_0 \subset S'$.
\end{proof}

\begin{remark} For a gapset filtration $F=F(\sigma, e)$ of multiplicity $m$, there is a strong connection between its exponent vector $e \in \N^{m-1}$ and the Kunz coordinates of the associated numerical semigroup $S(\sigma,e)$. 
\end{remark}
Indeed, let $S$ be a numerical semigroup of multiplicity $m$. Recall that the Ap\'ery set of $S$ is $\text{Ap}(S)=\{x\in\, |\, x-m\not\in S\}$. By Lemma 1.4 of~\cite{RGGJ1}, we have $\text{Ap}(S)=\{0=w(0),w(1),\dots,w(m-1)\}$ where $w(i)$ is the smallest element of $S$ which is congruent to $i$ modulo $m$.
Hence for $i\in[0,m-1]$ there exist $k_i\in\N$ such that $w(i)=i+mk_i$. 
The integers $k_1,\dots,k_{m-1}$ are the \emph{Kunz coordinates} of $S$.
From \eqref{SF}, we obtain that the smallest element of $S(\sigma,e)$ which is congruent to $\sigma(i)$ modulo $m$ is $\sigma(i)+m(e_0+\dots+e_{i-1})$. Hence for all $i\in[1,m-1]$, we have
$$
k_{\sigma(i)}=e_0+\dots+e_{i-1}.
$$

\subsection{The insertions maps $f_i$}

Let $m \ge 3$ and let $F=(F_0,\dots, F_t)$ be an $m$-filtration, i.e. with
$$
[1,m-1]=F_0 \supseteq F_1 \supseteq \dots \supseteq F_t.
$$
Let $g=\sum_{j=0}^t |F_j|$ be the genus of $F$. Given $i \in [1,m-1]$, we wish to insert $i$ in $F$ so as to end up with an $m$-filtration of genus $g+1$. There is only one way to do this, namely to insert $i$ in the first $F_j$ for which $i \notin F_j$. More formally, we define $f_i(F)$ as follows:
\begin{itemize}
\item If $i \in F_s \setminus F_{s+1}$ for some $s \le t-1$, then $f_i(F) = (F'_0, \dots, F'_t)$
where
$$
F'_j = \left\{
\begin{array}{lcl}
F_j & \textrm{if} & j \not= s+1, \\
F_{s+1} \sqcup \{i\} & \textrm{if} & j = s+1.
\end{array}
\right.
$$ 
\item If $i \in F_t$, then $f_i(F) = (F_0, \dots, F_t, F_{t+1})$ where $F_{t+1}=\{i\}$.
\end{itemize}
By construction, for all $i \in [1,m-1]$, we have that $f_i(F)$ is an $m$-filtration of genus $g+1$.

\smallskip
One delicate question is the following. If $F$ is a \emph{gapset} filtration of multiplicity $m$, for which $i \in [1,m-1]$ does it hold that $f_i(F)$ remains a gapset filtration? This question was successfully addressed in Section~\ref{section m=3} for $m=3$. 

\section{The case $m=4$}

We now use the above tools to characterize all gapset filtrations of multiplicity $m=4$ and to derive a counting-free proof of the inequality $n_{g+1, 4} \ge n_{g, 4}$ for all $g \ge 0$.

\smallskip
Let $F$ be a gapset filtration of multiplicity $m=4$. By Proposition~\ref{compact}, there exists $\sigma \in \mathfrak{S}_3$ and $e=(a,b,c) \in \N^3$ with $a \ge 1$ such that $F=F(\sigma,e)$. Moreover, Theorem~\ref{thrm} gives the exact conditions for $S(\sigma,e)$ to be a numerical semigroup, i.e. for $F(\sigma,e)$ to be a gapset filtration. This yields the following characterization, where the six elements of $\mathfrak{S}_3$ are displayed in window notation.

\begin{theorem}
\label{thm m=4}
The gapset filtrations of multiplicity $m=4$ are exactly the filtrations $F = F(\sigma,e)$ given in the table below, with $\sigma \in \mathfrak{S}_3$ and $e=(a,b,c) \in \N^3$ such that $a \ge 1$ and subject to the conditions below: 
\begin{equation}
\label{table}
\begin{array}{l|c|l}
\text{$\sigma \in \mathfrak{S}_3$ }& F = F(\sigma,e) &\emph{conditions on } a,b,c \\ 
\hline
(1,2,3)&(123)^a(23)^b(3)^c& b\leq a, \ c\leq a \\ 
(1,3,2)&(123)^a(23)^b(2)^c& b+c\leq a\\ 
(2,1,3)&(123)^a(13)^b(3)^c& c\leq a\\ 
(2,3,1)&(123)^a(13)^b(1)^c& c\leq a+1\\ 
(3,1,2)&(123)^a(12)^b(2)^c& b+c\leq a+1, \ c\leq a+b\\ 
(3,2,1)&(123)^a(12)^b(1)^c& b\leq a+1, \ c\leq a+1
\end{array}
\end{equation}
\end{theorem}
\begin{proof} Consider for instance the case $\sigma=(1,3,2)$. 
We have $\sigma(1)+\sigma(1)=\sigma(3)$ and $\sigma(2)+\sigma(2)=\sigma(3)+m$. Hence, by Theorem~\ref{thrm}, the conditions on $e=(a,b,c)$ for $S=S(\sigma,e)$ to be a numerical semigroup, i.e. for $F=F(\sigma,e)$ to be a gapset filtration, are exactly $b+c\leq a$ and $c\leq a+b+1$. Since the latter condition is implied by the former one, it may be ignored. We end up with the sole condition $b+c\leq a$, as stated in the table. The proof in the five other cases is again a straightforward application of Theorem~\ref{thrm} and is left to the reader.
\end{proof}

\begin{corollary}
\label{CI4} For all $g \ge 0$, there is an explicit injection
$$\fc(g,4) \longrightarrow \fc(g+1,4).$$
In particular, we have $n_{g+1,4} \ge n_{g,4}$ for all $g \ge 0$.
\end{corollary}

\begin{proof} The statement is trivial for $g \le 2$ since $\fc(g,4) = \emptyset$ in this case. Assume now $g \ge 3$. Let $F \in \fc(g,4)$ be a gapset filtration of genus $g$. Write $F=F(\sigma,e)$ for some $\sigma \in \mathfrak{S}_3$ and $e=(a,b,c) \in \N^3$ with $a \ge 1$. For $i=1,3$, consider the $4$-filtrations $F'=f_1(F)$ and $F''=f_3(F)$ of genus $g+1$ obtained by the insertion maps $f_1$ and $f_3$, respectively. Then $F'=F(\sigma,e')$ where
$$
e'=\begin{cases}
(a+1,b-1,c)&\text{if $\sigma\in\{(1,2,3),(1,3,2)\}$,}\\
(a,b+1,c-1)&\text{if $\sigma\in\{(2,1,3),(3,1,2)\}$,}\\
(a,b,c+1)&\text{if $\sigma\in\{(2,3,1),(3,2,1)\}$.}
\end{cases}
$$
It follows from \eqref{table} that $F'$ fails to be a gapset filtration, i.e. $F' \notin \fc(g+1,4)$, if and only  $\sigma = (2,3,1)$ or $(3,2,1)$ and $e=(a,b,a+1)$. This corresponds to $F$ being one of 
$$
(123)^a(13)^b(1)^{a+1}\ \text{or}\ (123)^a(23)^b(1)^{a+1}.
$$
Here $g=3a+2b+a+1=4a+2b+1$, whence $g$ is odd. In particular, if $g \not\equiv 1 \bmod 2$, then $F'$ is always a gapset filtration. We conclude that, whenever $g$ is \emph{even}, then $f_1$ yields a well-defined injection 
$$
\fc(g,4) \longrightarrow \fc(g+1,4).
$$
Let us now turn to $F''=f_3(F)$. Then $F''=F(\sigma, e'')$ where $e''$ is easily described by a table similar to \eqref{table}. Omitting details, it follows that $F''$ fails to be a gapset filtration if and only $F$ is one of
$$
(123)^a(13)^b(3)^{a}\ \text{or}\ (123)^a(23)^b(3)^{a}.
$$
In this case we have $g=3a+2b+a=4a+2b$, which is even. We conclude that whenever $g$ is \emph{odd}, then $f_3$ yields a well-defined injection 
$$
\fc(g,4) \longrightarrow \fc(g+1,4).
$$
This concludes the proof of the corollary.
\end{proof}

\subsection{Concluding remark}
We have shown that for $m=3$ and $4$, an injection $\fc(g,m) \longrightarrow \fc(g+1,m)$ is provided by one of the insertion maps $f_i$, where $i \in [1,m-1]$ depends on the class of $g$ modulo $3$ and $2$, respectively. 

Unfortunately, for any given $m \ge 5$, this is no longer true in general. That is, one should not expect that for each $g \ge 1$, an injection $\fc(g,m) \longrightarrow \fc(g+1,m)$ will be provided by just one of the insertion maps $f_i$. Constructing such injections for all $m,g$ remains open at the time of writing.


\end{document}